\documentclass[11pt]{article}
\usepackage{amsmath, amsfonts, amsthm,amssymb, color, hyperref, enumerate, extarrows, cite}

\newtheorem{theorem}{Theorem}
\newtheorem{lem}{Lemma}

\newtheorem{definition}[lem]{Definition}
\newtheorem*{notation*}{Notation}
{\theoremstyle{definition}
\newtheorem{example}[lem]{Example}}
\newtheorem{pro}[lem]{Proposition}
\newtheorem{corollary}[lem]{Corollary}

\numberwithin{equation}{section}
{\theoremstyle{definition}
\newtheorem{remark}[lem]{Remark}}
\usepackage[T1]{fontenc}

\hypersetup{
	colorlinks   = true,
	citecolor    = magenta}

\usepackage[titletoc]{appendix}

\newcommand{\ssubset}{\subset\subset}

\newcommand{\wt}{\widetilde} 
\newcommand{\core}{\mathfrak{C}}
\newcommand{\sn}[1]{S_{{\mathcal{N}}_{#1}}}
\newcommand{\snt}[1]{S_{\widetilde{\mathcal{N}}_{#1}}}

\newcommand{\cnt}[1]{\widetilde{\mathcal{N}}_{#1}}
\newcommand{\cu}{\mathcal{U}}

\newcommand{\C}{{\mathbb C}}

\begin{document}

\title{\vspace{-1.2cm} \bf Modifications of the Levi core\rm}

\author{Tanuj Gupta, Emil J.~Straube, John N.~Treuer}
\date{}

\maketitle

\begin{abstract}
We construct a family of subdistributions of the Levi core $\mathfrak{C}(\mathcal{N})$ called modified Levi cores $\{\mathcal{M}\mathfrak{C}_{\mathcal{A}}\}_{\mathcal{A}}$ indexed over closed distributions $\mathcal{A}$ that contain the Levi null distribution $\mathcal{N}$ and are contained in the complex tangent bundle $T^{1, 0}b\Omega$ of a smooth bounded pseudoconvex domain $\Omega$.  We show that Catlin's Property ($P$) holds on $b\Omega$ if and only if Property ($P$) holds on the support of one, and hence all, of the modified Levi cores.  In $\mathbb{C}^2$, all of the modified Levi cores coincide.  For a smooth bounded pseudoconvex complete Hartogs domain in $\mathbb{C}^2$ that satisfies Property ($P$), we show that its modified Levi core is trivial. This contrasts with $\mathfrak{C}(\mathcal{N})$, which can be nontrivial for such domains.

\end{abstract}

\renewcommand{\thefootnote}{\fnsymbol{footnote}}
\footnotetext{\hspace*{-7mm} 
\begin{tabular}{@{}r@{}p{16.5cm}@{}}
& Keywords. $\overline{\partial}$-Neumann problem, compactness, Property ($P$), Levi core, plurisubharmonic functions\\
& Mathematics Subject Classification. Primary 32W05, 32U05
\end{tabular}

\noindent\thanks{All three authors are partially supported by NSF grant DMS-2247175. The third author is also supported by an AMS-Simons travel grant.}

\noindent August 28, 2023
}

\section{Introduction}\label{intro}

Weakly pseudoconvex points in the boundary of a smooth bounded pseudoconvex domain may be obstacles to regularity properties of the $\overline{\partial}$--Neumann operator on the domain, but need not be. For example, it has been known since \cite{Sibony87b} that if the weakly pseudoconvex boundary points lie in a submanifold $\Sigma$ of the boundary of holomorphic dimension zero, then the boundary satisfies Property ($P$), and the $\overline{\partial}$--Neumann operator is compact and so in particular is globally regular. (This fact is implicit in \cite{Catlin82}.) A submanifold $\Sigma$ of the boundary is said to have holomorphic dimension zero if its complex tangent space has trivial intersection with the Levi null distribution $\mathcal{N}$, \cite{Sibony87b}. So roughly speaking, for regularity properties of the $\overline{\partial}$-Neumann operator, it matters whether or not the set of  weakly pseudoconvex points propagates along directions in $T^{1,0}(\Sigma)$ that are also Levi null directions. Points where this does not happen should not be of concern. 

When the set of weakly pseudoconvex points is not nicely contained in a smooth submanifold of the boundary, one can consider the Zariski tangent space to this set. The essential points (the points of concern) are now those where this tangent space shares a complex direction with the Levi null space at the point. Only these latter points, and the corresponding complex directions, are retained. This gives a new subdistribution of the the Levi null distribution, and the procedure of only retaining essential points can be repeated with this new distribution replacing the Levi null distribution. This idea comes from \cite{Dall'AraMongodi21}, where the authors recently introduced this idea. Inductively, one obtains a decreasing family of distributions that, although transfinite at the outset, must stabilize at a countable ordinal. The construction is analogous to the procedure in set theory of successively taking Cantor-Bendixson derivatives of a closed set in $\mathbb{R}^n$ in order to obtain its maximal perfect subset (see \cite{Kechris95}). The authors in \cite{Dall'AraMongodi21} called this stable distribution the Levi core of the boundary and proved many interesting properties. In particular, they showed that when the support of the Levi core is empty (i.e. there are no essential weakly pseudoconvex points left), the $\overline{\partial}$--Neumann operator is exactly regular. The same authors investigated further properties of the Levi core in \cite{Dall'AraMongodi23}.

The construction described in the previous paragraph takes into account at each step the differential geometry of the set of essential points, but does not take advantage of the local potential theory of this set to discharge points and thus reduce the size of the support of the core. Indeed, a domain may satisfy Property ($P$) yet have a Levi core whose support is big, for example has positive measure on the boundary. Such an example, a complete Hartogs domain in $\mathbb{C}^{2}$, was constructed in \cite{Tr22}. We discuss this example further in Section \ref{HartogsinC2}; see also \cite{Dall'AraMongodi23} for similar constructions. In Section \ref{modified} we propose a family of modifications of the Levi core construction from \cite{Dall'AraMongodi21} that seeks to remedy this situation and produces cores with smaller supports. In particular, in the above example, the support of the modified Levi core is empty (in $\mathbb{C}^{2}$, our family reduces to a single core). In general, we prove that if the support of one of our modified Levi cores has Property ($P$), then the boundary of the domain also has Property ($P$) (Theorem \ref{the main result.  Property (P) holds on the support of the new core if and only if it holds on the boundary of the domain}); this generalizes the analogous statement, shown in \cite{Tr22}, for the Levi core $\mathfrak{C}(\mathcal{N})$ from \cite{Dall'AraMongodi21}. Note that if the boundary has Property ($P$), then trivially the supports of the cores have it, since they are compact subsets of the boundary; in particular, the supports in our family either all have Property ($P$), or none do. Theorem \ref{the main result.  Property (P) holds on the support of the new core if and only if it holds on the boundary of the domain} is proved in Section \ref{main theorem}; Section \ref{prep} contains auxiliary lemmas needed in the proof of Theorem \ref{the main result.  Property (P) holds on the support of the new core if and only if it holds on the boundary of the domain}.

Our construction involves deciding whether or not certain local projections of the set of weakly pseudoconvex points onto subspaces of the complex tangent space (at a point) have Property ($P$) as subsets of that subspace. Because Property ($P$) is satisfactorily understood only in dimension one, in Section \ref{An alternative family of modified cores}, we also describe an alternative construction that only considers projections onto complex lines, so that one then only deals with planar sets (see Corollary \ref{doublymodified}). This alternative produces a less refined tool (the support of the resulting cores is bigger), but one that is easier to work with.

This alternative construction notwithstanding, what we do here is most useful in $\mathbb{C}^{2}$. In this case, there is only one modified Levi core for each domain. In Section \ref{HartogsinC2}, we show that for complete Hartogs domains in $\mathbb{C}^{2}$, the construction is optimal: the boundary satisfies Property ($P$) if and only if the support of the modified Levi core is empty (Proposition \ref{hartogs}). In addition, we revisit a class of examples from \cite{Tr22} whose Levi core was shown there to have positive measure in the boundary, and so in particular to have the maximal possible Hausdorff dimension three. We show that for these domains, the modified Levi core is trivial.

\section{A family of modifications of the Levi core}\label{modified}

Let $\Omega\subseteq\C^n$ be a bounded pseudoconvex domain. Given a distribution $\mathcal{D} = \cup_{p \in b\Omega}\mathcal{D}_p$ of $b\Omega$, with $\mathcal{D}_{p}$ a subspace of $T_{p}b\Omega$, let $S_{\mathcal{D}}$ denote the support of $\mathcal{D}$; that is,
$$
S_{\mathcal{D}} = \left\{p \in b\Omega: \mathcal{D}_p \neq \{0\}\right\}.
$$
$\mathcal{D}$ is said to be closed if it is closed in $Tb\Omega$. We denote the Levi null distribution by $\mathcal{N}$, so $\mathcal{N}_{p}$ is the null space of the Levi form at $p$.

A compact set $K$ in $\mathbb{C}^{n}$ is said to satisfy Property ($P$) if the following holds: for every $M>0$ there exists a $C^2$ plurisubharmonic function $\lambda_{M}$ in a neighborhood $U_{M}$ of $K$ with $0\leq \lambda_{M}\leq 1$ and with $\sum_{j,k=1}^{n}\frac{\partial^{2}\lambda_{M}}{\partial z_{j}\partial\overline{z_{k}}}(z)u_{j}\overline{u_{k}} \geq M|u|^{2}$, for all $u\in\mathbb{C}^{n}$ and $z\in U_{M}$. This notion goes back to \cite{Catlin82} when $K$ is the boundary of a domain and in general to \cite{Sibony87b} (under the name $B$--regularity). Its relevance stems from the fact that Property ($P$) of the boundary of a domain is the most widely applicable sufficient condition for compactness of the $\overline{\partial}$--Neumann operator on the domain; see \cite{Straube10}, chapter 4, for background and details on these topics. 

The construction of the Levi core (modified or not) requires the notion of the tangent distribution to an arbitrary subset $A\subseteq b\Omega$. We follow \cite{Dall'AraMongodi21}, Definition 2.5. At a point $p\in b\Omega$, a real vector $v_{p}\in T_{p}(b\Omega)$ belongs to the fiber $T_{p}A$ if and only if $v_{p}f=0$ for all $f\in C^{\infty}(b\Omega)$ with $f|_{A}\equiv 0$. We set $T^{1,0}_{p}A = T_{p}A\cap JT_{p}A$, where $J$ is the complex structure map (i.e. multiplication by $i$) in $\mathbb{C}^{n}$.

We now describe the algorithm to construct the modified Levi cores. The idea is to take into account the pluripotential theory along complex tangent directions of the support of a distribution (starting with $\mathcal{N}$) that are also Levi null directions. That is, we will declare a point negligible (see Definition \ref{neg1}) if locally there is a projection onto a specified subspace between the Levi null space and the complex tangent space whose image satisfies Property ($P$). The examples in \cite{Tr22} mentioned in the introduction are in $\mathbb{C}^{2}$. In this case, the relevant projections at a weakly pseudoconvex point $p\in b\Omega$ are onto $T_p^{1,0}(b\Omega)=\mathcal{N}_{p}$. In the case of $\mathbb{C}^{n}$, $n \geq 2$, $\mathcal{N}_{p}$ may be a nontrivial proper subspace of $ T^{1,0}_{p}(b\Omega)$, and one has to choose the subspace between $\mathcal{N}_p$ and $T_p^{1,0}b\Omega$ to project on. We do this by specifying a closed distribution $\mathcal{A}$ with $\mathcal{N}\subseteq\mathcal{A}\subseteq T^{1,0}(b\Omega)$, and projecting onto $\mathcal{A}_{p}$.

The construction is analogous to \cite{Dall'AraMongodi21}, as described in the introduction. Various elementary properties of ordinal numbers that are used in the sequel may be found, for example, in \cite{Halmos74}. We begin with a definition.  
\begin{definition}\label{neg1}
     Given closed distributions $\mathcal{A}$ and $\mathcal{D}$ on $b\Omega$ such that $\mathcal{D} \subseteq \mathcal{N} \subseteq \mathcal{A} \subseteq T^{1,0}b\Omega$, a point $p\in S_{\mathcal{D}}$ is said to be negligible with respect to $\mathcal{A}$ and $\mathcal{D}$ if there are a neighborhood $U\subseteq\C^n$ of $p$ and a linear projection $\pi_{p}\,:\,\mathbb{C}^{n}\rightarrow \mathcal{A}_{p}$
    such that for any neighborhood $V\ssubset U$ of $p$, $\pi_p(\overline{V}\cap S_{\mathcal{D}})$ has Property ($P$) in $\mathcal{A}_{p}$.
   \end{definition}

\smallskip

Set $\wt{\mathcal{N}}_{0, \mathcal{A}} = \wt{\mathcal{N}}_0 = \mathcal{N}$. Denote by $S_{0, \mathcal{A}}$ the complement of the negligible points with respect to $\widetilde{\mathcal{N}}_{0}$ and $\mathcal{A}$, i.e.,
\begin{equation}\label{def1} 
S_{0, \mathcal{A}} = \snt{0} \setminus \{ p \in \snt{0}\ :\ \hbox{$p$ is negligible with respect to $\widetilde{\mathcal{N}}_{0}$ and $\mathcal{A}$}  \}\;.
\end{equation}
Define 
\begin{equation}\label{def2}
(\mathcal{N}_{0, \mathcal{A}})_p := \begin{cases}
    \{0\} & p \not\in \overline{S_{0, \mathcal{A}}} \\
    (\wt{\mathcal{N}}_{0, \mathcal{A}})_p & p \in \overline{S_{0, \mathcal{A}}}.
\end{cases}
\end{equation}
Inductively, we define the higher order distributions. For a successor ordinal $\alpha$, define
\begin{equation}\label{def3}
\cnt{\alpha, \mathcal{A}} = \mathcal{N}_{\alpha -1, \mathcal{A}} \cap T^{1, 0}S_{\mathcal{N}_{\alpha - 1}, \mathcal{A}}.
\end{equation}
For a limit ordinal $\alpha$, define
\begin{equation}\label{def4}
\wt{\mathcal{N}}_{\alpha, \mathcal{A}} = \bigcap\limits_{\beta < \alpha} \wt{\mathcal{N}}_{\beta, \mathcal{A}}.
\end{equation}
Denote by $S_{\alpha, \mathcal{A}}$ the complement of the negligible points with respect to $\mathcal{A}$ and $\widetilde{\mathcal{N}}_{\alpha, \mathcal{A}}$, i.e.,
\begin{equation}\label{Salpha}
S_{\alpha, \mathcal{A}} = \snt{\alpha, \mathcal{A}} \setminus \{ p \in \snt{\alpha, \mathcal{A}}\ :\ \hbox{$p$ is negligible with respect to $\widetilde{\mathcal{N}}_{\alpha, \mathcal{A}}$ and $\mathcal{A}$} \}\;.
\end{equation}
Define 
\begin{equation}\label{defineN}
(\mathcal{N}_{\alpha, \mathcal{A}})_p := \begin{cases}
    \{0\} & p \not\in \overline{S_{\alpha, \mathcal{A}}} \\
    (\wt{\mathcal{N}}_{\alpha, \mathcal{A}})_p & p \in \overline{S_{\alpha, \mathcal{A}}}.
\end{cases}
\end{equation}
From \eqref{def1}-\eqref{defineN}, for any ordinal $\alpha$,
\begin{equation}\label{def intertwined Ns}
\widetilde{\mathcal{N}}_{\alpha, \mathcal{A}} \supseteq \mathcal{N}_{\alpha, \mathcal{A}} \supseteq \widetilde{\mathcal{N}}_{\alpha + 1, \mathcal{A}} \supseteq \mathcal{N}_{\alpha + 1, \mathcal{A}}.
\end{equation}
When the closed distribution $\mathcal{A}$ is clear from context, we will suppress the subscript $\mathcal{A}$ for these distributions and sets and write instead $\wt{\mathcal{N}}_{\alpha}, \mathcal{N}_{\alpha}$ or $S_{\alpha}$. In Lemmas \ref{N_alpha is a closed distribution}, \ref{stabilization}, \ref{rewriting the definition of derived distributions}-\ref{lemma about main decomposition of boundary}, the particular choice of closed distribution $\mathcal{A}$ with $\mathcal{N} \subseteq \mathcal{A} \subseteq T^{1, 0}b\Omega$ is not used in their proofs.  For clarity, the lemmas and their proofs are stated without reference to $\mathcal{A}$ and the abbreviated notation $\wt{\mathcal{N}}_{\alpha}, \mathcal{N}_{\alpha}$ or $S_{\alpha}$ is used.  In Section \ref{HartogsinC2}, all domains will be in $\mathbb{C}^2$, so for all weakly pseudoconvex boundary points $\mathcal{N}_p = \mathcal{A}_p = T^{1, 0}_pb\Omega$.  Hence in that section it will not be necessary to specify the distribution $\mathcal{A}$.
\begin{lem}\label{N_alpha is a closed distribution}
    For all ordinals $\alpha$, $\wt{\mathcal{N}_{\alpha}}$ and $\mathcal{N}_\alpha$ are closed distributions and $\sn{\alpha} = \overline{S_\alpha}$. Moreover, $\{\widetilde{\mathcal{N}_{\alpha}}\}_{\alpha}$, $\{S_{\alpha}\}_{\alpha}$, and $\{\mathcal{N}_{\alpha}\}_{\alpha}$ are monotone transfinite sequences.
\end{lem}

\begin{proof}
    We proceed by transfinite induction.  $\wt{\mathcal{N}_0}$ is a closed distribution because it is equal to the Levi null distribution.  It is clear from the definition that $\sn{0} = \overline{S_0}$. We now show that $\mathcal{N}_0$ is a closed distribution. Suppose $\{(p_n, v_n)\} \subset \mathcal{N}_0$ converges to $(p_{\infty}, v_{\infty})$. If infinitely many $v_n$'s are zero, then $v_{\infty} = 0$.  Thus, $(p_{\infty}, v_{\infty}) \in \mathcal{N}_0$.  On the other hand, if only finitely many $v_n's$ are zero, then passing to a subsequence, we can assume that $v_n \neq 0$ for all $n$.  Then $p_n \in \sn{0} = \overline{S_0}$, which implies that $p_{\infty} \in \overline{S_0}$. Moreover, since $p_n \in \sn{0}$, we have
    $$
    v_n \in (\cnt{0})_{p_n} = (\mathcal{N}_0)_{p_n}.
    $$
Since $(p_n, v_n) \in \cnt{0}$ and $\cnt{0}$ is a closed distribution, we get $(p_{\infty}, v_{\infty}) \in \cnt{0}$.   Thus,
    $$
    v_{\infty} \in (\cnt{0})_{p_{\infty}} = (\mathcal{N}_0)_{p_{\infty}},
    $$
    where the last equality follows because $p_{\infty} \in \sn{0} = \overline{S_0}$.  Thus, $(p_{\infty}, v_{\infty}) \in \mathcal{N}_0$. 
Suppose now that $\mathcal{N}_{\alpha}$ and $\wt{\mathcal{N}}_{\alpha}$ are known to be closed distributions.  Then $\wt{\mathcal{N}}_{\alpha + 1} = \mathcal{N}_{\alpha} \cap T^{1, 0}\sn{\alpha}$ is also closed because it is the intersection of two closed distributions.  The proof that $\sn{\alpha + 1} = \overline{S_{\alpha + 1}}$ and $\mathcal{N}_{\alpha + 1}$ is a closed distribution is the same is the same as the proof for the base case except that the subscript 0 is replaced by $\alpha + 1$.  
Next assume that for a limit ordinal $\alpha$, for all $\beta < \alpha$, it is known that $\mathcal{N}_{\beta}$ and $\wt{\mathcal{N}}_{\beta}$ are closed distributions.  Then $\wt{\mathcal{N}}_{\alpha}$ is closed because it is the intersection of closed distributions. The proof that $\sn{\alpha} = \overline{S_{\alpha}}$ and $\mathcal{N}_{\alpha}$ is a closed distribution is again the same as the proof for the base case except that the subscript 0 is now replaced by $\alpha$.
    
As for the monotonicity claims, the one for $\widetilde{\mathcal{N}}_{\alpha}$ again follows from a straightforward transfinite induction argument. Then $S_{\widetilde{\mathcal{N}}_{\alpha}}$ is likewise monotone, and the assertion for $S_{\alpha}$ follows from definition \eqref{Salpha} together with the fact that if a point $p\in S_{\widetilde{\mathcal{N}}_{\beta}}$ is negligible with respect to $\widetilde{\mathcal{N}_{\alpha}}$, then it is also negligible with respect to $\widetilde{\mathcal{N}_{\beta}}$ for $\beta > \alpha$. Finally, monotonicity of $\mathcal{N}_{\alpha}$ follows from that of $\widetilde{\mathcal{N}_{\alpha}}$ and $S_{\alpha}$ together with definition \eqref{defineN}.
\end{proof}

Lemma \ref{N_alpha is a closed distribution} shows that
$$
\widetilde{\mathcal{N}_0} \supseteq \widetilde{\mathcal{N}_1} \supseteq \widetilde{\mathcal{N}_2} \supseteq \cdots
$$
and
$$
\mathcal{N}_0 \supseteq \mathcal{N}_1 \supseteq \mathcal{N}_2 \supseteq \cdots
$$
are nested (transfinite) sequences of closed distributions. By \cite{Kechris95}, Theorem 6.9 (see also \cite{Dall'AraMongodi21}, Theorem 2.9), there is a countable ordinal $\gamma$ such that $\widetilde{\mathcal{N}_{\gamma}}=\widetilde{\mathcal{N}}_{\gamma+1}$.

\begin{lem}\label{stabilization}
For all ordinals $\alpha \geq \gamma$, we have
\begin{equation}\label{stable}
 \widetilde{\mathcal{N}_{\alpha}}=\widetilde{\mathcal{N}_{\gamma}}=\mathcal{N}_{\gamma}=\mathcal{N}_{\alpha}\;.
\end{equation}
\end{lem}
\noindent The last equality in \eqref{stable} is the one of main interest.
\begin{proof}
We have $\widetilde{\mathcal{N}}_{\gamma}=\widetilde{\mathcal{N}}_{\gamma+1}
= \mathcal{N}_{\gamma}\cap T^{1,0}S_{\mathcal{N}_{\gamma}}$ (by \ref{def3}) and $\mathcal{N}_{\gamma}\subseteq\widetilde{\mathcal{N}_{\gamma}}$ (by \ref{def intertwined Ns}); this implies $\mathcal{N}_{\gamma}=\widetilde{\mathcal{N}_{\gamma}}$. Definition \ref{Salpha} gives that $S_{\gamma+1}=S_{\gamma}$, whence $\mathcal{N}_{\gamma+1}=\mathcal{N}_{\gamma}$, from \eqref{defineN}. Now \eqref{stable} follows by transfinite induction. The step $\alpha \rightarrow (\alpha+1)$ is straightforward from the definitions. When $\alpha$ is a limit ordinal, the induction assumption gives $\widetilde{\mathcal{N}_{\alpha}}=\cap_{\beta<\alpha}\widetilde{\mathcal{N}_{\beta}} = \cap_{\gamma<\beta<\alpha}\widetilde{\mathcal{N}_{\beta}}=\widetilde{\mathcal{N}_{\gamma}}$. But then also $S_{\alpha}=S_{\gamma}$ (from \eqref{Salpha}), and, in turn, $\mathcal{N}_{\alpha}=\mathcal{N}_{\gamma}$.
\end{proof}
For a closed distribution $\mathcal{A}$ with $\mathcal{N} \subseteq \mathcal{A} \subseteq T^{1, 0}b\Omega$, we define the modified Levi core with respect to $\mathcal{A}$ by
\begin{equation}\label{the definition of the core}
\mathcal{M}\mathfrak{C}_{\mathcal{A}} := \mathcal{N}_{\gamma}\;,
\end{equation}
where $\gamma$ is the smallest countable ordinal that satisfies \eqref{stable}.  We shall suppress the subscript $\mathcal{A}$ when the particular choice of closed distribution $\mathcal{A}$ is clear from context or does not affect the proof of a lemma or theorem. The main result of the paper is that Property ($P$) on the boundary of the domain $\Omega$ is determined by the support of $\mathcal{M}\mathfrak{C}_{\mathcal{A}}$:
\begin{theorem}\label{the main result.  Property (P) holds on the support of the new core if and only if it holds on the boundary of the domain}
Let $\Omega$ be a smooth bounded pseudoconvex domain and $\mathcal{A}$ be a closed distribution such that $\mathcal{N} \subseteq \mathcal{A} \subseteq T^{1, 0}b\Omega$. Then $b\Omega$ satisfies Property ($P$) if and only if the support of $\mathcal{M}\mathfrak{C}_{\mathcal{A}}$ satisfies Property ($P$).
\end{theorem}
\noindent Since compact subsets of $b\Omega$ will inherit Property ($P$) from $b\Omega$ if it satisfies Property ($P$), whether or not the support of $\mathcal{M}\mathfrak{C}_{\mathcal{A}}$ satisfies Property ($P$) does not depend on the choice of closed distribution $\mathcal{A}$.  Consequently, the following corollary follows:

\begin{corollary}\label{corollary - property p holds for any closed distribution between the levi null space and the complex tangent space}
   Let $\Omega$ be a smooth bounded pseudoconvex domain and $\mathcal{A}$ and $\mathcal{B}$ be closed distributions such that $\mathcal{N} \subseteq \mathcal{A}, \mathcal{B} \subseteq T^{1, 0}b\Omega$. If the support of $\mathcal{M}\mathfrak{C}_{\mathcal{A}}$ satisfies Property ($P$), then the support of $\mathcal{M}\mathfrak{C}_{\mathcal{B}}$ satisfies Property ($P$) as well.
\end{corollary}

\begin{remark}
    The core of the Levi distribution $\core(\mathcal{N})$ defined by Dall'Ara and Mongodi \cite{Dall'AraMongodi21} is the first repeated distribution in the decreasing transfinite sequence 
    $$
    \mathcal{N} \supseteq \mathcal{N}^{1} \supseteq \cdots \supseteq\mathcal{N}^{\alpha} \supseteq \cdots,
    $$
    where $\mathcal{N}^{\alpha}$ is defined as in \cite{Dall'AraMongodi21}.  By comparing the algorithm presented in this paper with the one in \cite{Dall'AraMongodi21}, it is immediate that for any closed distribution $\mathcal{A}$ with $\mathcal{N} \subseteq \mathcal{A} \subseteq T^{1, 0}b\Omega$, for all ordinals $\alpha$, $ \mathcal{N}_{\alpha} \subseteq \mathcal{N}^{\alpha}$.  Thus, $\mathcal{M}\mathfrak{C}_{\mathcal{A}} \subseteq \core(\mathcal{N})$ and $S_{\mathcal{M}\mathfrak{C}_{\mathcal{A}}} \subseteq S_{\core(\mathcal{N})}$ (see the proof of Lemma \ref{stronger2}). In Example \ref{bigcore} in Section \ref{HartogsinC2}, we give examples that show that $S_{\mathcal{M}\mathfrak{C}_{\mathcal{A}}}$ may be strictly smaller than $S_{\core(\mathcal{N})}$.  
\end{remark}

\begin{remark}
    In \cite{Tr22}, it is proved that Property ($P$) holds on $b\Omega$ if Property ($P$) holds on $S_{\core(\mathcal{N})}$. Since $S_{\mathcal{M}\mathfrak{C}_{\mathcal{A}}} \subseteq S_{\mathfrak{C}(\mathcal{N})}$, if Property ($P$) holds on $S_{\mathfrak{C}(\mathcal{N})}$, it also holds on $S_{\mathcal{M}\mathfrak{C}_{\mathcal{A}}}$. Thus, Theorem \ref{the main result.  Property (P) holds on the support of the new core if and only if it holds on the boundary of the domain} improves the result in \cite{Tr22}. 
\end{remark}

\section{Additional preparatory lemmas}\label{prep}
In this section we prove additional lemmas that will be needed for the proof of Theorem \ref{the main result.  Property (P) holds on the support of the new core if and only if it holds on the boundary of the domain}. These lemmas are similar in nature to Lemmas 2.5, 2.7, 2.8 in \cite{Tr22}, which are for the (unmodified) Levi core.
\begin{lem}\label{rewriting the definition of derived distributions}
Let $\Omega$ be a smooth bounded pseudoconvex domain and $p \in b\Omega$.  For any ordinal $\alpha$,
\begin{equation}\label{Characterization of tilde N alpha in terms of N tilde 0}
\wt{\mathcal{N}}_{\alpha} = \begin{cases}
\wt{\mathcal{N}}_0 & \alpha = 0
\\
\wt{\mathcal{N}}_0 \cap T^{1,0}\sn{\alpha - 1} & \alpha \hbox{ successor ordinal} 
\\
 \wt{\mathcal{N}}_0 \cap \bigcap\limits_{\beta < \alpha} T^{1, 0}S_{\mathcal{N}_{\beta}}  & \alpha \hbox{ limit ordinal.}
\end{cases}
\end{equation}
\end{lem}
\begin{proof}
We prove this claim using transfinite induction.  The base case follows by definition.  Assume that the claim holds for the ordinal $\alpha$.  We will show it is true for its successor.  Suppose first that $\alpha$ is a successor ordinal.  By definition,
$$
\wt{\mathcal{N}}_{\alpha + 1} = \mathcal{N}_{\alpha} \cap T^{1, 0}S_{\mathcal{N}_{\alpha}}.
$$
If $p \in S_{\mathcal{N}_{\alpha}}$, then $(\mathcal{N}_{\alpha})_p = (\wt{\mathcal{N}}_{\alpha})_p$.  Consequently,
\begin{align*}
(\widetilde{\mathcal{N}}_{\alpha + 1})_p &= (\wt{\mathcal{N}}_{\alpha})_p \cap  T_p^{1, 0} \sn{\alpha}
\\
&= (\wt{\mathcal{N}}_{0})_p \cap T_p^{1, 0}\sn{\alpha - 1} \cap T_p^{1, 0}\sn{\alpha}
\\
&= (\cnt{0})_p \cap T_p^{1, 0}\sn{\alpha}
\end{align*}
If $p \not\in \sn{\alpha} = \overline{S_{\alpha}} = \overline{\sn{\alpha}}$, then $T_p^{1, 0}\sn{\alpha} = \{0\}$.  Hence,
$$
(\wt{\mathcal{N}}_{\alpha + 1})_p = (\wt{\mathcal{N}}_{\alpha})_p \cap T_p^{1, 0}\sn{\alpha} = \{0\} = (\wt{\mathcal{N}}_0)_p \cap T_p^{1, 0}\sn{\alpha}.
$$
If $\alpha$ is a limit ordinal, the argument is analogous, replacing $T^{1,0}_{p}S_{\mathcal{N}_{\alpha-1}}$ by $\cap_{\beta<\alpha}T^{1,0}_{p}S_{\mathcal{N}_{\beta}}$. Thus, we have shown that if \eqref{Characterization of tilde N alpha in terms of N tilde 0} holds for an ordinal $\alpha$, then it holds for its successor.  

It remains to show that if $\alpha$ is a limit ordinal and \eqref{Characterization of tilde N alpha in terms of N tilde 0} holds for all ordinals $\beta < \alpha$, then \eqref{Characterization of tilde N alpha in terms of N tilde 0} holds for the ordinal $\alpha$.  Let $A = \{\beta: \beta < \alpha \hbox{ is a successor ordinal}\}.$ Notice that
$$
\wt{\mathcal{N}}_{\alpha} = \bigcap_{\beta<\alpha}\widetilde{\mathcal{N}_{\beta}} = \bigcap\limits_{\beta \in A} \wt{\mathcal{N}}_{\beta} = \bigcap\limits_{\beta \in A} (\wt{\mathcal{N}}_0 \cap T^{1,0}\sn{\beta}) = \wt{\mathcal{N}}_0 \cap \bigcap_{\beta \in A} T^{1, 0}\sn{\beta} = \wt{\mathcal{N}}_0 \cap \bigcap_{\beta < \alpha} T^{1, 0}\sn{\beta}.
$$
\end{proof}
\begin{lem}\label{limit ordinal support intersection lemma}
    For any limit ordinal $\alpha$, we have
    $$ \snt{\alpha} = \bigcap_{\beta<\alpha} \snt{\beta}.$$
\end{lem}
\begin{proof}
    Clearly $\snt{\alpha} \subseteq \bigcap_{\beta<\alpha} \snt{\beta}$. On the other hand if $p\in \bigcap_{\beta<\alpha} \snt{\beta}$, then consider
    $$ (\cu\cnt{\beta})_p := (\cnt{\beta})_p\cap\left\{ X\in T_p^{1,0}b\Omega\ :\ \left\| X \right\| = 1 \right\}. $$
    Then $\{ (\cu\cnt{\beta})_p \}_{\beta<\alpha}$ is a decreasing transfinite sequence of non-empty compact subsets of the unit sphere bundle in $T_p^{1,0}b\Omega$. So there exists 
    $$ X_\infty\in\bigcap_{\beta<\alpha} (\cu\cnt{\beta})_p\subseteq\bigcap_{\beta<\alpha} (\cnt{\beta})_p = (\cnt{\alpha})_p $$
    with $\|X_\infty\|=1$. Hence $p\in\snt{\alpha}$ and the lemma is proved. 
\end{proof}
\begin{lem}\label{lemma about main decomposition of boundary}
    Let $\Omega$ be a smooth bounded pseudoconvex domain. The boundary $b\Omega$ decomposes as
    \begin{equation}\label{equation of main decomposition of boundary}
        b\Omega = K_{-1} \cup \left( \bigcup_{0\le \beta \leq \gamma} \snt{\beta}\setminus \snt{\beta+1} \right) \cup S_{\mathcal{M}\mathfrak{C}},
    \end{equation}
    where $K_{-1}$ is the set of strictly pseudoconvex points of $b\Omega$, and $\gamma$ is a countable ordinal such that $\mathcal{M}\mathfrak{C} = \mathcal{N_{\gamma}} = \mathcal{N}_{\gamma + 1}$.
\end{lem}
\begin{proof}
    We will first show by transfinite induction that for any ordinal $\alpha$,
    \begin{equation}\label{decomposition of weakly-pseudoconvex points}
        \snt{0} = \bigcup_{0\le\beta<\alpha} \left( \snt{\beta}\setminus\snt{\beta+1}\right)\cup \snt{\alpha}.
    \end{equation}
We omit the base case and successor ordinal cases and show only the limit ordinal case.  
Suppose that $\alpha$ is a limit ordinal and that \eqref{decomposition of weakly-pseudoconvex points} holds for all ordinals $\delta$ with $1 \leq \delta<\alpha$, i.e.,
    \begin{equation}\label{claim holding for delta}
        \snt{0} = \bigcup_{0\le\beta<\delta} \left( \snt{\beta}\setminus\snt{\beta+1}\right)\cup \snt{\delta}.
    \end{equation}
We will show that \eqref{claim holding for delta} also holds for $\delta = \alpha$. It is clear that
$$ \snt{0} \supseteq \bigcup_{0\le\beta<\alpha} \left( \snt{\beta}\setminus\snt{\beta+1}\right)\cup \snt{\alpha}. $$
On the other hand, let $p\in\snt{0}$ be such that $p\not\in \bigcup\limits_{0\le\beta<\alpha} \left( \snt{\beta}\setminus\snt{\beta+1}\right)$. Then for $1 \leq \delta<\alpha$, $p\not\in \bigcup\limits_{0\le\beta<\delta} \left( \snt{\beta}\setminus\snt{\beta+1}\right)$. By \eqref{claim holding for delta}, $p\in\snt{\delta}$ for all $1 \leq \delta<\alpha$. By Lemma \ref{limit ordinal support intersection lemma}, 
    $$
    p\in\bigcap_{1\le\delta<\alpha} \snt{\delta} = \bigcap_{0\le\delta<\alpha} \snt{\delta} = \snt{\alpha}.
    $$
This proves \eqref{decomposition of weakly-pseudoconvex points}. Now apply \eqref{decomposition of weakly-pseudoconvex points} with $\alpha = \gamma + 1$ and observe that $S_{\widetilde{\mathcal{N}}_{\gamma + 1}}=S_{\mathcal{N}_{\gamma}}=S_{\mathcal{M}\mathfrak{C}}$ (by \eqref{def intertwined Ns}; note that by assumption, $\mathcal{N}_{\gamma}=\mathcal{N}_{\gamma+1}$).
Therefore, \eqref{equation of main decomposition of boundary} holds and the proof is complete.
\end{proof}

\section{Proof of Theorem \ref{the main result.  Property (P) holds on the support of the new core if and only if it holds on the boundary of the domain}}\label{main theorem}

In this section, we prove the main result of the paper, Theorem \ref{the main result.  Property (P) holds on the support of the new core if and only if it holds on the boundary of the domain}.
\begin{proof}
For ease of notation, we will suppress the subscript $\mathcal{A}$ on $\mathcal{M}\mathfrak{C}_{\mathcal{A}}$ and other distributions and sets. If $b\Omega$ satisfies Property ($P$), then any compact subset, in particular $S_{\mathcal{M}\mathfrak{C}}$ satisfies Property ($P$).  We now assume that $S_{\mathcal{M}\mathfrak{C}}$ satisfies Property ($P$) and show that $b\Omega$ satisfies Property ($P$). Let $\gamma$ be an ordinal such that $\mathcal{M}\mathfrak{C} = \mathcal{N}_{\gamma} = \mathcal{N}_{\gamma + 1}$. Since by Lemma \ref{lemma about main decomposition of boundary},
\begin{equation}\label{boundarydecomp}
b\Omega = K_{-1} \cup \left( \bigcup_{0\le \beta\le \gamma} \snt{\beta}\setminus \snt{\beta+1} \right) \cup S_{\mathcal{M}\mathfrak{C}} 
\end{equation}
is a countable union of sets, it suffices to show that each set in \eqref{boundarydecomp} can be covered by countably many compact subsets of $b\Omega$ that satisfy Property ($P$) (\cite{Sibony87b}, Proposition 1.9; see also \cite{Straube10}, Corollary 4.14). This is clear for $K_{-1}$ and $S_{\mathcal{M}\mathfrak{C}}$. Because $S_{\widetilde{\mathcal{N}}_{\beta+1}} \subseteq\sn{\beta}\subseteq S_{\widetilde{\mathcal{N}}_{\beta}}$, we have $S_{\widetilde{\mathcal{N}}_{\beta}}\setminus S_{\widetilde{\mathcal{N}}_{\beta+1}} = (\snt{\beta}\setminus\sn{\beta}) \cup (\sn{\beta}\setminus\snt{\beta + 1})$. We now show the required covers also exist for $\snt{\alpha}\setminus\sn{\alpha}$ and $\sn{\alpha}\setminus\snt{\alpha + 1}$, where $\alpha$ is an arbitrary ordinal. Without loss of generality, suppose below that these sets are nonempty.

Let $p \in \snt{\alpha}\setminus\sn{\alpha}$. Then $p\notin S_{\alpha}\subseteq\overline{S_{\alpha}}$, so $p$ is negligible with respect to $\mathcal{A}$ and $\widetilde{\mathcal{N}}_{\alpha}$. Select a coordinate neighborhood $U \subseteq \mathbb{C}^n$ centered at $p$ and, via a linear coordinate change, coordinates $(\xi_1,\ldots,\xi_n)$ that satisfy that $\mathcal{N}_{p}=\{\xi_{m+1}=\cdots=\xi_{n}=0\}$, $\mathcal{A}_p = \{\xi_{t + 1} = \cdots = \xi_n = 0\}$ for some $t \geq m$ and $\pi_{p}$ becomes the orthogonal projection onto the $(\xi_{1}, \ldots,\xi_{t})$--subspace $\mathcal{A}_{p}$. Since $p \not\in \overline{S_{\alpha}} = \sn{\alpha} = \overline{\sn{\alpha}}$, after shrinking $U$, one may suppose that $U \cap \sn{\alpha} = \emptyset$ and for any neighborhood $V \ssubset U$ of $p$ that $\pi_p(\overline{V}\cap\snt{\alpha})$ satisfies 
Property ($P$). We have used here that a complex linear change of coordinates which turns the projection $\pi_{p}$ into an orthogonal projection does not affect whether or not the projection of a set has Property ($P$). Fix a neighborhood $V_p \ssubset U$ of $p$ small enough such that for $\xi\in\overline{V_{p}}\cap b\Omega$, if $u\in\mathcal{N}_{\xi}$, then $\sum_{s=1}^{m}|u_{s}|^{2}\geq (1/2)|u|^{2}$. Such a neighborhood exists, otherwise, a sequence of points $(\xi)_{j}$ and unit vectors $u_{j}\in\mathcal{N}_{(\xi)_{j}}$ would lead to a unit vector $u\in\mathcal{N}_{p}$ with $\sum_{s=1}^{m}|u_{s}|^{2}\leq (1/2)|u|^{2}$, a contradiction. Rescaling $\rho$ if necessary, we may also assume that $\sum_{j=1}^{n}|(\partial\rho/\partial\xi_{j})|^{2}\geq (1/2)$ on $V_{p}$.

Now fix $M>0$. Since $\pi_p(\overline{V_p} \cap \snt{\alpha})$ satisfies Property ($P$), there exists a $C^2$-plurisubharmonic function $\mu$ on a neighborhood $W \subseteq \mathcal{A}_p$ of $\pi_p(\overline{V_p}\cap\snt{\alpha})$ with 
$$
0 \leq \mu \leq 1 \hbox{ and } \sum_{j, k  = 1}^t {\partial^2\mu \over \partial \xi_j\partial \overline{\xi}_k}(\xi)u_j\overline{u_k} \geq M|u|^2, \quad \xi \in W,\quad u \in \mathbb{C}^t.
$$
Set
\begin{equation}\label{definelambda}
\lambda(\xi_1,\ldots,\xi_n) = \mu(\xi_{1},\ldots,\xi_{t}) + A\rho(\xi_{1},\ldots, \xi_{n}) + B\rho^2(\xi_1,\ldots, \xi_n).
\end{equation}
$\lambda$ is defined in a neighborhood of $\overline{V_{p}}\cap S_{\widetilde{\mathcal{N}}_{\alpha}}$; $A$ and $B$ will be chosen later in order to produce a big Hessian.

Let $u\in \mathbb{C}^{n}$. We have for $\xi \in b\Omega$,
\begin{multline}\label{Hessian1}
 \sum_{j,k=1}^{n}\frac{\partial^{2}\lambda}{\partial\xi_{j}\partial\overline{\xi_{k}}}(\xi)u_{j}\overline{u_{k}} = \sum_{j,k=1}^{t}\frac{\partial^{2}\mu}{\partial\xi_{j}\partial\overline{\xi_{k}}}(\xi)u_{j}\overline{u_{k}} + A\sum_{j,k=1}^{n}\frac{\partial^{2}\rho}{\partial\xi_{j}\partial\overline{\xi_{k}}}(\xi)u_{j}\overline{u_{k}} + 2B\left|\sum_{j=1}^{n}\frac{\partial\rho}{\partial\xi_{j}}(\xi)u_{j}\right|^{2} \\
 \geq M\sum_{s=1}^{m}|u_{s}|^{2} + A\sum_{j,k=1}^{n}\frac{\partial^{2}\rho}{\partial\xi_{j}\partial\overline{\xi_{k}}}(\xi)u_{j}\overline{u_{k}} + 2B\left|\sum_{j=1}^{n}\frac{\partial\rho}{\partial\xi_{j}}(\xi)u_{j}\right|^{2}\;.
\end{multline}
We claim that for $A$ big enough (depending on $M$), the right hand side of \eqref{Hessian1} is at least $(M/3)|u|^{2}$ for $u\in T^{1,0}(b\Omega)$. Assume this does not hold. Then there exists a sequence $\{(q^{l},u^{l})\}_{l=1}^{\infty}$ with $q^{l} \in \overline{V_{p}}\cap S_{\widetilde{\mathcal{N}}_{\alpha}}$ and a unit vector $u^{l}\in T^{1,0}_{q^{l}}(b\Omega)$ such that 
\begin{equation}\label{Hessian2}
 M\sum_{s=1}^{m}|u^{l}_{s}|^{2} + l\sum_{j,k=1}^{n}\frac{\partial^{2}\rho}{\partial\xi_{j}\partial\overline{\xi_{k}}}(q^{l})u^{l}_{j}\overline{u^{l}_{k}}  < \frac{M}{3}\;.
\end{equation}
By compactness, there is a subsequence that converges to $(q,u)$, $q\in \overline{V_{p}}\cap S_{\widetilde{\mathcal{N}}_{\alpha}}$ and the unit vector $u\in T^{1,0}_{q}(b\Omega)$. Because $\Omega$ is pseudoconvex, the Hessian term in \eqref{Hessian2} is nonnegative for all $l$. This implies two things. First, $M\sum_{s=1}^{m}|u^{l}_{s}|^{2}< \frac{M}{3}$ for all $l$ and therefore $M\sum_{s=1}^{m}|u_{s}|^{2}\leq \frac{M}{3}$, that is $\sum_{s=1}^{m}|u_{s}|^{2}\leq \frac{1}{3}$. Second, $\sum_{j,k=1}^{n}\frac{\partial^{2}\rho}{\partial\xi_{j}\partial\overline{\xi_{k}}}(q^{l})u^{l}_{j}\overline{u^{l}_{k}}\rightarrow 0$ as $l\rightarrow\infty$, and therefore $u\in \mathcal{N}_{q}$. This contradicts our assumption on $V_{p}$. Note that for $A=A_{M}$ chosen as above, $M\sum_{s=1}^{m}|u_{s}|^{2} + A\sum_{j,k=1}^{n}\frac{\partial^{2}\rho}{\partial\xi_{j}\partial\overline{\xi_{k}}}(\xi)u_{j}\overline{u_{k}} \geq C_{M}|u|^{2}$ for some (possibly negative) constant $C_{M}$, $(\xi,u)\in (\overline{V_{p}}\cap S_{\widetilde{\mathcal{N}}_{\alpha}})\times \mathbb{C}^{n}$ (by homogeneity, $C_{M}$ is the minimum over $(\overline{V_{p}}\cap S_{\widetilde{\mathcal{N}}_{\alpha}})\times \mathbb{S}^{2n-1}$). An argument analogous to the previous one now shows that a constant $B=B_{M}$ can be chosen big enough so that the right hand side of \eqref{Hessian1} is at least $(1/3)M|u|^{2}$.
So on $\overline{V_p}\cap\snt{\alpha}$, the Hessian of $\lambda$ is at least $(M/3)$ and $0\leq \lambda \leq 1$. Since $M$ was arbitrarily chosen, $\overline{V_p}\cap\snt{\alpha}$ satisfies Property ($P$). The open sets $V_{p}$ cover $\snt{\alpha}\setminus\sn{\alpha}$. If $V_{p_{j}}$, $j\in\mathbb{N}$, is a countable subcover, then the sets $\overline{V_{p_{j}}}\cap\snt{\alpha}$ provide a countable cover of $\snt{\alpha}\setminus\sn{\alpha}$ by compact sets satisfying Property ($P$).

\smallskip

It remains to show that for any ordinal $\alpha$, $\sn{\alpha}\setminus\snt{\alpha + 1}$ can be covered by countably many compact sets that satisfy Property ($P$).  Let $p \in \sn{\alpha}\setminus\snt{\alpha + 1}$.  By Lemma \ref{rewriting the definition of derived distributions}, since $p \not\in \snt{\alpha + 1}$,
$$
\{0\} = (\wt{\mathcal{N}}_{\alpha + 1})_p = (\wt{\mathcal{N}}_0)_p \cap T_p^{1, 0}\sn{\alpha}.
$$
By Proposition 2.6(c) in \cite{Dall'AraMongodi21}, there exists a real embedded submanifold $V_p$ of $b\Omega$ such that 
$$
T_p^{1, 0}V_p = T_p^{1, 0}\sn{\alpha}
$$
and $V_p$ locally contains $\sn{\alpha}$ (hence $\sn{\alpha}\setminus\snt{\alpha + 1}$) near $p$; that is, there exists $r_p > 0$ such that
\begin{equation}\label{Vp locally contains Snalpha take away Sntildealpha plus 1}
V_p \supseteq (\sn{\alpha}\setminus \snt{\alpha + 1}) \cap \overline{\mathbb{B}(p, r_p)},
\end{equation}
where $\mathbb{B}(p, r_p) = \{z \in \mathbb{C}^n: |z - p| < r_p\}$. Notice that
$$
\{0\} = (\wt{\mathcal{N}}_{\alpha + 1})_p = (\wt{\mathcal{N}}_0)_p \cap T_p^{1, 0}V_p.
$$
Since $\wt{\mathcal{N}}_0 \cap T^{1, 0}V_p$ is a closed distribution, the dimension of the fibers is an upper semicontinuous function by Proposition 2.3 in \cite{Dall'AraMongodi21}. It follows that after shrinking $V_p$ and $r_p$ that for all points of $V_p$, the fibers of $\wt{\mathcal{N}}_0 \cap T^{1, 0}V_p$ are trivial. Let $U_p \subseteq V_p$ be an open set such that
$$
V_p \supseteq \overline{U_p} \supseteq U_p \supseteq (\sn{\alpha}\setminus \snt{\alpha + 1}) \cap \overline{\mathbb{B}(p, {r_p})}.
$$
By \cite{Sibony87b}, Proposition 1.12, (see also \cite{Straube10}, Proposition 4.15; this fact is also implicit in \cite{Catlin82}),  $\overline{U_p}$ satisfies Property ($P$). Let $\{U_{p_{j}}\}_{j\in\mathbb{N}}$ be a countable subcover of $\{U_{p}\}_{p\in(\sn{\alpha}\setminus\snt{\alpha + 1})}$ for $\sn{\alpha}\setminus\snt{\alpha + 1}$. Then $\{\overline{U_{p_{j}}}\}_{j\in\mathbb{N}}$ is a countable cover of $\sn{\alpha}\setminus \snt{\alpha + 1}$ by compact sets satisfying Property ($P$).

The proof of the theorem is now complete.
\end{proof}

\begin{remark}\label{onlyLevinullspace}
The first part of the proof of Theorem \ref{the main result.  Property (P) holds on the support of the new core if and only if it holds on the boundary of the domain} is analogous to the proof of (i) implies (ii) in Proposition 1 in \cite{Straube12}. This implication says that (in our context) the boundary of a domain satisfies Property ($P$) if it satisfies Property ($P$) on the nullspace of the Levi form. 
 
\end{remark}

\section{An alternative family of modified cores}\label{An alternative family of modified cores}

The algorithm to produce the modified Levi cores requires deciding whether or not certain sets in the complex tangent space at a point satisfy Property ($P$). Property ($P$) is reasonably well understood only in dimension one, where there are a number of equivalent characterizations (see Proposition 1.11 in \cite{Sibony87b}; Proposition 5 in \cite{FuStraube02}; Proposition 4.17 in \cite{Straube10}); in particular, a compact set in the plane satisfies Property ($P$) if and only if it has empty fine interior (see \cite{ArmitageGardiner02, Helms69} for information on the fine topology). In Section \ref{HartogsinC2}, we will use this last characterization to produce examples in $\mathbb{C}^{2}$ of domains with trivial modified Levi cores and nontrivial (unmodified) Levi cores. From this point of view, Theorem \ref{the main result.  Property (P) holds on the support of the new core if and only if it holds on the boundary of the domain} is most useful in $\mathbb{C}^{2}$: determining whether a weakly pseudoconvex point is negligible reduces to determining whether a compact set in $\mathbb{C}$ satisfies Property ($P$). Alternatively in $\mathbb{C}^n$, one can weaken the construction of the modified core so that one only has to deal with Property ($P$) for sets in the plane, at the expense of a (possibly) larger support of the core. To this end, the requirement in Definition \ref{neg1} for a point $p$ to be negligible with respect to the distributions $\mathcal{A}$ and $\mathcal{D}$ is made stronger (see Definition \ref{verynegligible} and Lemma \ref{stronger}), so there are fewer negligible points, and the support of the resulting core may become bigger. 
\begin{definition}\label{verynegligible}
Let $\mathcal{A}$ and $\mathcal{D}$ be closed distributions on $b\Omega$ such that $\mathcal{D} \subseteq \mathcal{N} \subseteq \mathcal{A} \subseteq T^{1, 0}b\Omega$, $p\in S_{\mathcal{D}}$, and $m:=\dim\mathcal{A}_{p}$. The point $p$ is said to be very negligible with respect to $\mathcal{A}$ and $\mathcal{D}$ if there are a neighborhood $U\subseteq\C^n$ of $p$, linearly independent complex lines $E_{1},\ldots, E_{m}\subseteq \mathcal{A}_{p}$, and parallel projections $\pi_{j}:\mathbb{C}^{n}\rightarrow E_{j}$, $1\leq j\leq m$, whose kernels are in general position, such that for any neighborhood $V\ssubset U$ of $p$, $\pi_j(\overline{V}\cap S_{\mathcal{D}})$ satisfies Property ($P$) for $1\le j\le m$. 
\end{definition}
\noindent We say that the kernels of the projections $\pi_{j}$, which are complex hyperplanes, are in general position if their intersection has dimension $(n-m)$ (equivalently, if the normals to these hyperplanes are linearly independent). This is the case for example when all $\pi_{j}$s are orthogonal projections. 
\begin{lem}\label{stronger}
Let $\mathcal{A}$ and $\mathcal{D}$ be distributions in $b\Omega$ as in Definition \ref{verynegligible}, $p\in S_{\mathcal{D}}$ and $m =\dim\mathcal{A}_{p}$. If $p$ is very negligible with respect to $\mathcal{A}$ and $\mathcal{D}$, then it is also negligible with respect to $\mathcal{A}$ and $\mathcal{D}$.
 \end{lem}
\begin{proof}
Note that $\ker(\pi_{1}+\cdots+\pi_{m})=\cap_{j=1}^{m}\ker\pi_{j}$ (the $E_{j}$s are linearly independent). Thus $\dim(\ker(\pi_{1}+\cdots+\pi_{m})) = (n-m)$. Since $Im(\pi_{1}+\cdots+\pi_{m})\subseteq \mathcal{A}_{p}$, it follows from dimension considerations that firstly $Im(\pi_{1}+\cdots+\pi_{m})= \mathcal{A}_{p}$, and secondly, $\ker(\pi_{1}+\cdots+\pi_{m})\cap\mathcal{A}_{p}=\{0\}$. Consequently we have the direct sum decomposition $\mathbb{C}^{n}=\mathcal{A}_{p}\oplus\ker(\pi_{1}+\cdots+\pi_{m})$. Denote the associated projection onto $\mathcal{A}_{p}$ by $\Pi$. Note that $\pi_{j}\Pi=\pi_{j}$, $1\leq j\leq m$. Now fix $M$. We identify $\mathcal{A}_{p}$ with $\mathbb{C}^{m}$, with coordinates $\xi = (\xi_{1},\ldots,\xi_{m})$. Then for each $j$, $1\leq j\leq m$, there is a $C^2$ subharmonic function $\mu_{j}$ in a neighborhood of $\pi_{j}(\overline{V}\cap S_{\mathcal{D}})$ with Laplacian (Hessian) at least $M$, and $0\leq \mu_{j}\leq 1$. The function $\mu(\xi):=\mu_{1}(\pi_{1}\xi)+\cdots+\mu_{m}(\pi_{m}\xi)$ is plurisubharmonic in a neighborhood of $\Pi(\overline{V}\cap S_{\mathcal{D}})$ and satisfies $0\leq \mu\leq m$. Moreover, $\sum_{j,k=1}^{m}\frac{\partial^{2}\mu}{\partial\xi_{j}\partial\overline{\xi_{k}}}u_{j}\overline{u_{k}} \geq M(|\pi_{1}u|^{2}+\cdots+|\pi_{m}u|^{2})\geq cM|u|^{2}$ for some positive constant $c$ that does not depend on $M$. Otherwise, there is a sequence $\{u_{n}\}_{n=1}^{\infty}$ of unit vectors with  $(|\pi_{1}u_{n}|^{2}+\cdots+|\pi_{m}u_{n}|^{2})\leq (1/n)|u_n|^{2}$; passing to the limit of a convergent subsequence yields a unit vector $u\in\cap_{j=1}^{m}\ker(\pi_{j}|_{\mathcal{A}_{p}})$, a contradiction. It follows that $\Pi(\overline{V}\cap S_{\mathcal{D}})$ has Property ($P$) in $\mathcal{A}_{p}$, that is, $p$ is negligible with respect to $\mathcal{A}$ and $\mathcal{D}$.
\end{proof}

Starting with the Levi null distribution, successive derived distributions can now be obtained in the same way as in Section \ref{modified} for the construction of $\mathcal{M}\mathfrak{C}_{\mathcal{A}}$, but using the notion of very negligible instead of negligible. Note that the `monotonicity property' of negligibility used there, namely that $p\in S_{\widetilde{\mathcal{N}}_{\beta, \mathcal{A}}}$ is negligible with respect to $\mathcal{A}$ and $\widetilde{\mathcal{N}}_{\beta, \mathcal{A}}$ if it is negligible with respect to $\mathcal{A}$ and $\widetilde{\mathcal{N}}_{\alpha, \mathcal{A}}$ and $\beta > \alpha$ also holds for the notion of very negligible. In particular, Lemma \ref{N_alpha is a closed distribution} remains true and we obtain a new core which we denote by $\mathcal{M}^{2}\mathfrak{C}_{\mathcal{A}}$.
\begin{lem}\label{stronger2}
$\mathcal{M}\mathfrak{C}_{\mathcal{A}}\subseteq\mathcal{M}^{2}\mathfrak{C}_{\mathcal{A}}$. In particular, $S_{\mathcal{M}\mathfrak{C}_{\mathcal{A}}}\subseteq S_{\mathcal{M}^{2}\mathfrak{C}_{\mathcal{A}}}$.
\end{lem}
\begin{proof}
Recall the definitions \eqref{def1}--\eqref{defineN}. We use a superscript hat to indicate the corresponding notions for the construction of $\mathcal{M}^{2}\mathfrak{C}_{\mathcal{A}}$. In the argument below, we will again suppress the subscript $\mathcal{A}$. Lemma \ref{stronger} gives that $S_{0}\subseteq \widehat{S_{0}}$ and therefore $\overline{S_{0}}\subseteq\overline{\widehat{S_{0}}}$. Definition \eqref{def2} then gives that $\mathcal{N}_{0}\subseteq\widehat{\mathcal{N}_{0}}$. We trivially have $\widetilde{\mathcal{N}_{0}}\subseteq\widehat{\widetilde{{\mathcal{N}_{0}}}}$ (they are both equal to $\mathcal{N})$. Now a transfinite induction argument, using definitions \eqref{def3}--\eqref{defineN} shows that for all ordinals $\alpha$, $\widetilde{\mathcal{N}_{\alpha}}\subseteq\widehat{\widetilde{{\mathcal{N}_{\alpha}}}}$ and $\mathcal{N}_{\alpha}\subseteq\widehat{\mathcal{N}_{\alpha}}$. Suppose that $\{\mathcal{N}_{\alpha}\}$ stabilizes at $\gamma_{1}$ ($\gamma_{1}$ from Definition \eqref{the definition of the core}) and $\{\widehat{\mathcal{N}_{\alpha}}\}$ stabilizes at $\gamma_{2}$ (Lemma \ref{stabilization} also holds for $\widehat{\mathcal{N}_{\alpha}}$). If $\gamma_{2}\leq \gamma_{1}$, then $\mathcal{M}\mathfrak{C}=\mathcal{N}_{\gamma_{1}}\subseteq\mathcal{N}_{\gamma_{2}}\subseteq\widehat{\mathcal{N}_{\gamma_{2}}}=\mathcal{M}^{2}\mathfrak{C}$. If $\gamma_{1}<\gamma_{2}$, then $\mathcal{M}\mathfrak{C}=\mathcal{N}_{\gamma_{1}}=\mathcal{N}_{\gamma_{2}}\subseteq\widehat{\mathcal{N}_{\gamma_{2}}}=\mathcal{M}^{2}\mathfrak{C}$. 
\end{proof}

Lemma \ref{stronger2} and Theorem \ref{the main result.  Property (P) holds on the support of the new core if and only if it holds on the boundary of the domain} immediately give the following corollary. 
\begin{corollary}\label{doublymodified}
Let $\Omega$ be a smooth bounded pseudoconvex domain and $\mathcal{A}$ a closed distribution on $b\Omega$ with $\mathcal{N}\subseteq\mathcal{A}\subseteq T^{1,0}(b\Omega)$. Then $b\Omega$ satisfies Property ($P$) if and only if the support of $\mathcal{M}^{2}\mathfrak{C}_{\mathcal{A}}$ satisfies Property ($P$). 
\end{corollary}
\noindent Corollary \ref{doublymodified} implies the following analogue of Corollary \ref{corollary - property p holds for any closed distribution between the levi null space and the complex tangent space}.

\begin{corollary}\label{corollary - property p holds for any closed distribution between the levi null space and the complex tangent space 2}
   Let $\Omega$ be a smooth bounded pseudoconvex domain and $\mathcal{A}$ and $\mathcal{B}$ be closed distributions such that $\mathcal{N} \subseteq \mathcal{A}, \mathcal{B} \subseteq T^{1, 0}b\Omega$. If the support of $\mathcal{M}^2\mathfrak{C}_{\mathcal{A}}$ satisfies Property ($P$), then the support of $\mathcal{M}^2\mathfrak{C}_{\mathcal{B}}$ satisfies Property ($P$) as well.
\end{corollary}

\smallskip

For domains in $\mathbb{C}^{2}$, $\mathcal{M}\mathfrak{C}_{\mathcal{A}}$ and $\mathcal{M}^{2}\mathfrak{C}_{\mathcal{A}}$ trivially coincide. But there are examples of interest also in higher dimensions where $\mathcal{M}\mathfrak{C}_{\mathcal{A}} = \mathcal{M}^{2}\mathfrak{C}_{\mathcal{A}}$; this happens trivially when $\mathcal{M}^{2}\mathfrak{C}_{\mathcal{A}}$ is trivial (in view of Lemma \ref{stronger2}). As examples, consider domains whose weakly pseudoconvex boundary points form a set of two--dimensional Hausdorff measure zero (shown to have Property ($P$) in \cite{Sibony87b} and \cite{Boas88}). Any projection onto a complex line then has measure zero (on the line) and therefore has Property ($P$) (see Proposition 5 in \cite{FuStraube02}; Proposition 4.17 in \cite{Straube10}; an elementary proof is in \cite{Boas88}). So $\mathcal{M}^{2}\mathfrak{C}_{\mathcal{A}}$ is trivial for all $\mathcal{A}$.

\section{Hartogs domains in $\mathbb{C}^{2}$}\label{HartogsinC2}

In $\mathbb{C}^{2}$, all of the modified Levi cores coincide; that is, for any two closed distributions $\mathcal{A}$ and $\mathcal{B}$ such that $\mathcal{N} \subseteq \mathcal{A}, \mathcal{B} \subseteq T^{1, 0}b\Omega$, 
$$
\mathcal{M}\mathfrak{C}_{\mathcal{A}} = \mathcal{M}\mathfrak{C}_{\mathcal{B}} = \mathcal{M}^2\mathfrak{C}_{\mathcal{A}} =  \mathcal{M}^2\mathfrak{C}_{\mathcal{B}}.
$$
In this section, we will use the notation $\mathcal{M}\mathfrak{C}$ for this modified core. We will show that on the class of smooth pseudoconvex complete Hartogs domains in $\mathbb{C}^{2}$, the boundary satisfies Property ($P$) if and only if the modified Levi core is trivial. In addition, in this class, there are domains that have a modified Levi core with trivial support and an unmodified Levi core with support of positive measure. 

\smallskip

The examples are smooth pseudoconvex complete Hartogs domains of the form
\begin{equation}\label{nick's example domain}
\Omega = \left\{(z, w)\in \mathbb{C}^{2}:\, z\in \mathbb{D},\, |w|^2 < e^{-\phi(z)}\right\}\;,
\end{equation}
where $\mathbb{D}$ is the unit disc in $\mathbb{C}$ and $\phi\in C^{\infty}(\mathbb{D})$ is subharmonic.
\begin{pro}\label{hartogs}
Let $\Omega$ be a smooth pseudoconvex complete Hartogs domain in $\mathbb{C}^{2}$ as in \eqref{nick's example domain}. Then the boundary of $\Omega$ satisfies Property ($P$) if and only if $\mathcal{M}\mathfrak{C}$ is trivial.
\end{pro}
The weakly pseudoconvex boundary points over the base $\mathbb{D}$ correspond to the points where $\Delta\varphi$ vanishes. Denote this set by $A$ and denote the weakly pseudoconvex boundary points in $b\mathbb{D}$ by $B$. Then 
\begin{equation}\label{nick's example weakly pseudoconvex points}
\snt{0} = \left\{\left(z, e^{-{\phi(z) \over 2} + i \theta}\right):\, z \in A,\, \theta \in [0, 2\pi]\right\} \cup \{(z,0): z\in B\}\;,
\end{equation}
and the Levi null distribution is given by
$$ 
(\cnt{0})_{(z, w)} = \left\{\begin{array}{cc}
                             sp_{\mathbb{C}}\left\{\partial_z - w\phi_z(z)\partial_w\right\}, & z\in A\;,\\
                             sp_{\mathbb{C}}\{\partial_{w}\}\;,&z\in B\;,
                            \end{array} \right.
$$
see Section 4.7 in \cite{Straube10}.  
\begin{proof}[Proof of Proposition \ref{hartogs}]
One direction follows from Theorem \ref{the main result.  Property (P) holds on the support of the new core if and only if it holds on the boundary of the domain}. The other direction is in part equivalent to the observation in \cite{Sibony87b}, Example on page 310. Note that if $K$ is a compact subset of $\mathbb{D}$, then the set $K\cap A$ must satisfy Property ($P$) if the boundary does; the proof is the same as that of Lemma 4.19 in \cite{Straube10}. Fix $p=(z_0,w_0)\in\snt{0}$, with $z_0\in K\cap A$ and $w_0 = e^{-\frac{\phi(z_0)}{2}+i\theta_0}$ for some $\theta_0 \in [0, 2\pi]$. Let 
$$ E_1 := T_p^{1,0}b\Omega = \text{span}_\C\left\{ \left( 1 , -w_{0}\frac{\partial\phi}{\partial z}(z_0) \right) \right\},$$
and denote by $L_{p}$ the corresponding complex tangent at $p$. Let $\pi(z, w) = (z, 0)$ and $\pi_{1}:\mathbb{C}^{2}\rightarrow L_{p}$ be the linear projection parallel to the $w$--axis. $\pi_{1}$ is given by
$$ 
\pi_{1}(z,w) = \left( z, w_0 - (z-z_0)w_0\dfrac{\partial\phi}{\partial z}(z_0) \right), $$
and so
\begin{align*}
    \pi_{1}(\snt{0}\cap \pi^{-1}(K)) &= \pi_{1}\left( \left\{ \left( z,e^{-\frac{\phi(z)}{2}+i\theta} \right)\ :\ z\in A\cap K, \theta\in[0,2\pi] \right\} \right)\\
    &=\left\{ \left( z,w_0 - (z-z_0)w_0\dfrac{\partial\phi}{\partial z}(z_0) \right)\ :\ z\in A\cap K \right\} \\
    &=\Phi_1(A\cap K),
\end{align*}
where $\Phi_1:\{(z,0)\ :\ z\in\C\}\to L_{p}$ is the biholomorphic map defined by
$$
\Phi_1(z,0) = \left( z,w_0 - (z-z_0)w_0\dfrac{\partial\phi}{\partial z}(z_0) \right).
$$
Since $A\cap K$ satisfies Property ($P$) in $\mathbb{C}$, so does $\Phi_1(A\cap K) = \pi_1(\snt{0}\cap\pi^{-1}(K))$ in $L_{p}$. Thus $(z_{0},w_{0})\notin S_{0}$, and $S_{0}\subseteq\overline{S_{0}}\subseteq B\subseteq b\mathbb{D}$. Also $S_{\mathcal{N}_{0}} = \overline{S_{0}}$, so that $T^{1,0}_{p}S_{\mathcal{N}_{0}}=\{0\}$ for all $p\in S_{\mathcal{N}_{0}}$. By definition \ref{def3}, $(\widetilde{\mathcal{N}_{1}})_{p}=(\mathcal{N}_{0})_{p}\cap T^{1,0}_{p}S_{\mathcal{N}_{0}}=\{0\}$ for all $p\in b\Omega$. By definition \ref{defineN}, we likewise have $(\mathcal{N}_{1})_{p}=\{0\}$ for all $p$, which in turn implies (via \ref{def3}) that $\widetilde{\mathcal{N}_{2}}$ is trivial. In particular $\widetilde{\mathcal{N}_{1}}=\widetilde{\mathcal{N}_{2}}$, that is $\widetilde{\mathcal{N}_{\alpha}}$ and $\mathcal{N}_{\alpha}$ stabilize at $\alpha=1$. Therefore, $\mathcal{M}\mathfrak{C}=\mathcal{N}_{1}$ is trivial.
\end{proof}

\begin{example}\label{bigcore}
We now construct examples, also in $\mathbb{C}^{2}$, where $\mathcal{M}\mathfrak{C}_{\mathcal{A}}\subsetneq\mathfrak{C}(\mathcal{N})$. It is shown on page 96 in \cite{Straube10} (see the discussion following the proof of Lemma 4.19) that if $K$ is an arbitrary compact subset of $\mathbb{D}$, there exists a Hartogs domain $\Omega$ as in \eqref{nick's example domain} whose weakly pseudoconvex boundary points are the boundary points $(z,w)$ with $z\in K$. When $K \subset \{z \in \mathbb{C}:\, |z| < {1 \over 2}\}$ is a square Cantor set (a Cartesian product of two Cantor sets, scaled to fit inside disc of radius ${1 \over 2}$), $K$ has empty fine interior, so that the boundary of $\Omega$ satisfies Property ($P$). Indeed, a fine interior point of $K$ would give rise to arbitrarily small circles centered at the point and contained in $K$ (\cite{Helms69}, Theorem 10.14, \cite{ArmitageGardiner02}, Theorem 7.3.9). The projection of such a circle onto the real coordinate axes would contain an interval, contradicting the fact that these projections are Cantor sets. Proposition \ref{hartogs} now implies that the modified Levi core of $\Omega$ is trivial. The (unmodified) Levi core, however, has support equal to the full set of weakly pseudoconvex points, that is $S_{\core(\mathcal{N})} = \snt{0}$, see \cite{Tr22}. This set may be `big' in terms of Hausdorff dimension: following \cite{Tr22}, we note that if we choose the Cantor sets with positive measure (in $\mathbb{R}$), then the square Cantor set will have positive measure in the plane. The rotational symmetry of $\Omega$ implies that the set of weakly pseudoconvex boundary points, hence the support of the (unmodified) Levi core, has positive measure in the boundary and thus also has the maximal possible Hausdorff dimension of three. See also \cite{Dall'AraMongodi23} for further constructions of domains with non-trivial Levi cores using Cantor sets. 
\end{example}

\begin{remark}\label{Reinhardt}

The argument in the proof of Proposition \ref{hartogs} that shows that the weakly pseudoconvex boundary points over the base are negligible with respect to $T^{1,0}(b\Omega)$ and $\mathcal{N}$ also works for complete Hartogs domains in $\mathbb{C}^{n}$, with fairly routine modifications of the computations from Section 4.7 in \cite{Straube10}. Consider now a smooth bounded pseudoconvex complete  Reinhardt domain $\Omega$ in $\mathbb{C}^{n}$ (centered at the origin). Completeness implies that $\Omega$ contains the origin. Suppose $p$ is a weakly pseudoconvex boundary point. Then at least one of its coordinates is nonzero; after a permutation of the coordinates, we may assume that the $z_{n}$--coordinate of $p$ is nonzero. Such a permutation preserves invariance under the canonical $\mathbb{T}^{n}$--action, and we may consider $\Omega$ as a complete Hartogs domain over a base in $\mathbb{C}^{n-1}$. If the boundary of $\Omega$ satisfies Property ($P$), it therefore follows that $p$ is negligible with respect to $T^{1,0}(b\Omega)$ and $\mathcal{N}$. We thus obtain a partial analogue of Proposition \ref{hartogs} for complete Reinhardt domains in $\mathbb{C}^{n}$, as follows. Let $\Omega$ be a smooth bounded pseudoconvex complete Reinhardt domain in $\mathbb{C}^{n}$. Then $b\Omega$ satisfies Property ($P$) if and only if $\mathcal{M}\mathfrak{C}_{T^{1,0}(b\Omega)}$ is trivial. This argument breaks down when $\mathcal{N}\subseteq\mathcal{A}\subsetneq T^{1,0}(b\Omega)$.

\end{remark}

\providecommand{\bysame}{\leavevmode\hbox to3em{\hrulefill}\thinspace}

\fontsize{11}{9}\selectfont

\vspace{0.5cm}

\noindent tanujgupta17@tamu.edu;

 \vspace{0.2 cm}

\noindent Department of Mathematics, Texas A\&M University, College Station, TX 77843-3368, USA

\vspace{0.6 cm}

\noindent e-straube@tamu.edu;

 \vspace{0.2 cm}

\noindent Department of Mathematics, Texas A\&M University, College Station, TX 77843-3368, USA

\vspace{0.6 cm}

\noindent jtreuer@tamu.edu; jtreuer@ucsd.edu

 \vspace{0.2 cm}

\noindent Department of Mathematics, Texas A\&M University, College Station, TX 77843-3368, USA

\vspace{0.2 cm}
\noindent Department of Mathematics, University of California San Diego, La Jolla, CA 92093, USA


\begin{thebibliography}{CJW20}

\bibitem{ArmitageGardiner02}
Armitage, David H. and Gardiner Stephen J., \emph{Classical Potential Theory}, Springer Monographs in Mathematics, Springer, 2002.

\bibitem{Boas88}
Boas, Harold P., Small sets of infinite type are benign for the $\overline{\partial}$-Neumann problem, \emph{Proc. Amer. Math. Soc.} \textbf{103}, Nr. 2 (1988), 569--578.

\bibitem{Catlin82}
Catlin, David W., Global regularity of the $\overline{\partial}$--Neumann problem. In \emph{Complex analysis of several variables} (Madison 1982), Proc. Sympos. Pure Math., 41, Amer. Math. Soc., Providence, RI, 1984, 39--49.

\bibitem{Dall'AraMongodi21}
Dall'Ara, Gian Maria, Mongodi, Samuele, The core of the Levi distribution, \emph{Journal de l'\'{E}cole polytechnique. Math\'{e}matiques}, \textbf{10} (2023), 1047--1095.

\bibitem{Dall'AraMongodi23}
\bysame, Remarks on the Levi core, preprint 2023,  	arXiv:2305.17439.


\bibitem{FuStraube02}
Fu, Siqi and Straube, Emil J., Semi-classical analysis of Schr\"{o}dinger operators and compactness in the $\overline{\partial}$-Neumann problem, \emph{J. Math. Anal. Appl.} \textbf{271} (2002), 267--282. Correction in \emph{ibid} \textbf{280} (2003), 195--196.


\bibitem{Halmos74}
Halmos, Paul R., \emph{Naive Set Theory}, Reprint of the 1960 edition, Undergraduate Texts in Mathematics. Springer, New York-Heidelberg, 1974.

\bibitem{Helms69}
Helms, L.~L., \emph{Introduction to Potential Theory}. Wiley-Interscience, 1969.

\bibitem{Kechris95}
Kechris, Alexander S., \emph{Classical Descriptive Set Theory}, Graduate Texts in Mathematics, \textbf{156}, Springer New York, 1995.

\bibitem{Sibony87b}
Sibony, Nessim, Une classe de domaines pseudoconvexes, \emph{Duke
Math. J.} \textbf{55} , Nr. 2 (1987), 299--319.

\bibitem{Straube10}
Straube, Emil J., \emph{Lectures on the $L^2$-Sobolev theory of the $\overline{\partial}$-Neumann problem}, ESI Lectures in Mathematics and Physics, European Mathematical Society (EMS), Z\"{u}rich 2010.

\bibitem{Straube12}
\bysame, The complex Green operator on CR--submanifolds of hypersurface type: compactness, \emph{Trans. Amer. Math. Soc.} {\bf 364}, Nr. 8 (2012), 4107--4125.

\bibitem{Tr22}
Treuer, John N., Sufficient condition for compactness of the $\overline{\partial}$--Neumann operator using the Levi core, \emph{Proc. Amer. Math. Soc.}, forthcoming. 

\end{thebibliography}
\end{document}